\documentclass{article}

\usepackage{epsfig}
 \usepackage{epstopdf}
\usepackage[T1]{fontenc}
\usepackage{geometry}
\usepackage{amsbsy,amsmath,latexsym,amsfonts, epsfig, color, authblk, amssymb, graphics, bm}
\usepackage{epsf,slidesec,epic,eepic}
\usepackage{fancybox}
\usepackage{fancyhdr}
\usepackage{setspace}
\usepackage{nccmath}
\usepackage{color}
\usepackage[colorlinks,linkcolor=black,anchorcolor=black,citecolor=black,hyperindex,CJKbookmarks]{hyperref}
\newenvironment{proof}{{\noindent \bf Proof.}}{\hfill$\Box$\medskip}

\newtheorem{theorem}{Theorem}[section]
\newtheorem{corollary}[theorem]{Corollary}
\newtheorem{lemma}[theorem]{Lemma}
\newtheorem{proposition}[theorem]{Proposition}
\newtheorem{definition}{Definition}[section]

\newtheorem{remark}{Remark}

\def \b{\beta}

\def \G{\Gamma}

\def \d{\delta}

\def \R{\mathbb{R}}

\def \N{\mathbb{N}}

\def \A{\mathcal{A}}

\def \M{\mathcal{M}}

\begin{document}

\title{Diophantine approximation and run-length function on $\beta$-expansions
}
\author{  Lixuan Zheng\\
\small \it Department of Mathematics, South China University of Technology, Guangzhou 510640, P.R. China\\
\& \it  LAMA, Universit\'{e} Paris-Est Cr\'{e}teil, 61 av G\'{e}n\'{e}ral de Gaulle, 94010, Cr\'{e}teil, France}
\date{11th July,2018}
\date{}
\maketitle
\begin{center}
\begin{minipage}{120mm}{\small {\bf Abstract.} For any $\beta > 1$, denoted by $r_n(x,\beta)$ the maximal length of consecutive zeros amongst the first $n$ digits of the $\beta$-expansion of $x\in[0,1]$. The limit superior (respectively limit inferior) of $\frac{r_n(x,\beta)}{n}$ is linked to the classical Diophantine approximation (respectively uniform Diophantine approximation). We obtain the Hausdorff dimension of the level set $$E_{a,b}=\left\{x \in [0,1]: \liminf_{n\rightarrow \infty}\frac{r_n(x,\beta)}{n}=a,\ \limsup_{n\rightarrow \infty}\frac{r_n(x,\beta)}{n}=b\right\}\ (0\leq a\leq b\leq1).$$ Furthermore, we show that the extremely divergent set $E_{0,1}$ which is of zero Hausdorff dimension is, however, residual.  The same problems in the parameter space are also examined.}
\end{minipage}
\end{center}

\vskip0.5cm

\section{Introduction}
Let $\beta>1$ be a real number. The \emph{$\beta$-transformation} on $[0,1]$ is  defined by  $$T_{\beta}(x) = \beta x-\lfloor\beta x\rfloor,$$where $\lfloor \xi\rfloor$ means the integer part of $\xi$. It is well-known (see \cite{R}) that, every real number $x \in [0, 1]$ can be uniquely expanded as a series
\begin{equation}\label{1.1}
x=\frac{\varepsilon_1(x,\beta)}{\b}+\cdots+\frac{\varepsilon_n(x,\beta)}{\beta^n}+\cdots,
\end{equation}
where $\varepsilon_n(x,\beta)=\lfloor \beta T_{\beta}^{n-1}(x)\rfloor$ for all $n\geq 1$. We call $\varepsilon_n(x,\beta)$ \emph{the $n$-th digit of $x$} and $\varepsilon(x,\beta):=(\varepsilon_1(x,\beta), \ldots,\varepsilon_n(x,\beta),\ldots)$ the \emph{$\beta$-expansion of $x$}.

For each $x\in [0,1]$ and $n\geq 1$, the {\it run-length function} $r_n(x,\beta)$ is defined to be the maximal length of consecutive zeros amongst the prefix $(\varepsilon_1(x,\beta),\ldots,\varepsilon_n(x,\beta))$, i.e.,\ $$r_n(x,\beta)=\max\{1\leq j\leq n: \varepsilon_{i+1}(x,\beta)=\cdots=\varepsilon_{i+j}(x,\beta)=0 \ {\rm for\ some}\ 0\leq i \leq n-j\}.$$ If such $j$ does not exist, we set $r_n(x,\beta)=0$. In 1970, Erd\"{o}s and R\'{e}nyi \cite{ER} showed that for Lebesgue almost all $x\in[0,1]$, we have
\begin{equation}\label{rn}
\lim_{n\rightarrow \infty}\frac{r_n(x,2)}{\log_2n}=1.
\end{equation}
The result of Erd\"{o}s and R\'{e}nyi \cite{ER} has been extended to the general case $\beta>1$ by Tong, Yu and Zhao \cite{TYZ}. Ma, Wen and Wen \cite{MW} showed that the exceptional set of points violating (\ref{rn}) is of full Hausdorff dimension. Let $\mathcal{E}$ denote the set of increasing functions $\varphi:\N\rightarrow (0,+\infty)$ satisfying $\lim\limits_{n\rightarrow \infty }\varphi(n)=+\infty$ and $\limsup\limits_{n\rightarrow \infty}\frac{\varphi(n)}{n}=0$.  For every $0\leq a\leq b\leq \infty$ and any function $\varphi\in \mathcal{E}$, define
$$E_{a,b}^{\varphi}:=E_{a,b}^{\varphi}(\beta)=\left\{x \in [0,1]:\liminf\limits_{n\rightarrow \infty}\frac{r_n(x,\beta)}{\varphi(n)}=a,\ \limsup\limits_{n\rightarrow \infty}\frac{r_n(x,\beta)}{\varphi(n)}=b\right\}.$$
The set $E_{a,b}^{\varphi}$ has been proved to have full Hausdorff dimension by Li and Wu (see \cite{LW1,LW2}) for the case $\beta=2$ and by Zheng, Wu and Li \cite{ZWL} for the general case $\beta>1$.

Remark that the above $\mathcal{E}$ does not contain the function $\varphi(n)=n$. In fact, the asymptotic behavior of $\frac{r_n(x,\beta)}{n}$ is directly related to the Diophantine approximation of $\beta$-expansions. For all $x\in[0,1]$, Amou and Bugeaud \cite{B} defined the exponent $v_{\beta}(x)$ to be the supremum of the real numbers $v$ for which the equation $$T_\beta^n x\leq \beta^{-nv}$$ has infinitely many positive integer $n$.  Bugeaud and Liao \cite{AB} defined the exponent $\hat{v}_{\beta}(x)$ to be the supremum of the real numbers $\hat{v}$ for which, for all $N\gg 1$, there is a solution with $1\leq n\leq N$, such that $$T_\beta^n x\leq \beta^{-Nv}.$$ We will see (Lemmas \ref{ship} and \ref{ship1}) that for all $0<a<1,\ 0<b<1$, $$\liminf\limits_{n\rightarrow \infty}\frac{r_n(x,\beta)}{n}=a\ \Leftrightarrow\ \hat{v}_{\beta}(x)=\frac{a}{1-a}\quad {\rm and} \quad \limsup\limits_{n\rightarrow \infty}\frac{r_n(x,\beta)}{n}=b\ \Leftrightarrow\ v_{\beta}(x)=\frac{b}{1-b}.$$

For all $0\leq a\leq b\leq1$, let
\begin{equation}\label{ab1}
E_{a,b}:=E_{a,b}(\beta)=\left\{x \in [0,1]:\liminf_{n\rightarrow \infty}\frac{r_n(x,\beta)}{n}=a,\ \limsup_{n\rightarrow \infty}\frac{r_n(x,\beta)}{n}=b\right\}.
\end{equation}
Denote the Hausdorff dimension by $\dim_{\rm H}$. For more information about the Hausdorff dimension, we refer to \cite{FE}. We establish the following theorem.
\begin{theorem}\label{ab}
The set $E_{0,0}$ has full Lebesgue measure. If $\frac{b}{1+b}< a\leq 1,\ 0< b \leq 1$, then $E_{a,b}=\emptyset$. Otherwise, we have  $$\dim_{\rm{H}}E_{a,b}=1-\frac{b^2(1-a)}{b-a}.$$
\end{theorem}


Let $0\leq a\leq 1$ and  $0\leq b\leq 1$. We can further study the level sets
$$E_a:=E_a(\beta)=\left\{x \in [0,1]:\liminf_{n\rightarrow \infty}\frac{r_n(x,\beta)}{n}=a\right\}$$and
\begin{equation}\label{bb}
F_b:=F_b(\beta)=\left\{x \in [0,1]:\limsup_{n\rightarrow \infty}\frac{r_n(x,\beta)}{n}=b\right\}.
\end{equation}
Using Theorem \ref{ab}, we obtain the following results of the Hausdorff dimensions of $E_a$ and $F_b$.
\begin{corollary}\label{a}
(1) When $0\leq a\leq \frac{1}{2}$, we have $$\dim_{\rm{H}}E_a={(1-2a)}^2.$$ Otherwise, $E_{a}=\emptyset$.

(2) For all $0\leq b\leq1$, we have $$\dim_{\rm{H}}F_b=1-b.$$
\end{corollary}

We remak that the statement (2) of Theorem \ref{a} was also been obtained in \cite[Theorem 1.1]{LL} (see \cite{Z} for the case $\beta=2$).

A set $R$ is called \emph{residual} if its complement is meager (i.e.,\ of the first category). In a complete metric space, a set is residual if it contains a dense $G_\d$ set, i.e.,\ a countable intersection of open dense sets (see \cite{JC}).  Similar to the results of \cite{LW1,LW2,ZWL}, the set of extremely divergent points is residual, and thus is large in the sense of topology.
\begin{theorem}\label{res}
The set $E_{0,1}$ is residual in $[0,1]$.
\end{theorem}

It is worth noting that the set $E_{0,1}$ is negligible with respect to the Lebesgue measure and Hausdorff dimension. However, the sets considered in \cite{LW1,LW2,ZWL} have Hausdorff dimension $1$. Let $\varphi\in \mathcal{E}$. Since the intersection of two residual sets is still residual, by combining Theorems \ref{res} and \cite[Theorem 1.2] {ZWL}, we deduce that the smaller set $$E:=E(\varphi,\beta)=\left\{x \in [0,1]:\liminf_{n\rightarrow \infty}\frac{r_n(x,\beta)}{\varphi(n)}=0,\ \limsup_{n\rightarrow \infty}\frac{r_n(x,\beta)}{n}=1\right\}$$ is also residual in $[0,1]$.

\bigskip

The $\beta$-expansion of $1$ completely characterizes all of the admissible words in the $\beta$-dynamical system (see Theorem \ref{P} in Section 2 for more details). We also study the run-length function $r_n(\beta)$ of the $\beta$-expansion of $1$ as $\beta$ varies in the parameter space $\{\beta\in\R: \beta>1\}$, i.e.,
$$r_n(\beta) = \max\{1\leq j \leq n : \varepsilon_{i+1}(\beta) =\cdots= \varepsilon_{i+j}(\beta) = 0 {\rm\ for\ some\ } 0 \leq i\leq n-j\}.$$ There are some results on $r_n(\beta)$ which are similar to those of $r_n(x,\beta)$. In \cite{HTY}, Hu, Tong and Yu proved that for Lebesgue almost all $1<\beta<2$, we have
\begin{equation}\label{1}
\lim_{n\rightarrow \infty}\frac{r_n(\beta)}{\log_\beta n}=1.
\end{equation}
Cao and Chen \cite{CC} showed that for any $\varphi \in \mathcal{E}$ and for all $0\leq a \leq b\leq +\infty$, the set $$\left\{\beta \in (1,2):\liminf\limits_{n\rightarrow \infty}\frac{r_n(\beta)}{\varphi(n)}=a,\ \limsup\limits_{n\rightarrow \infty}\frac{r_n(\beta)}{\varphi(n)}=b\right\}$$ is of full Hausdorff dimension. Remark that the results of \cite{HTY} and \cite{CC} can be easily generalized to the whole parameter space $\{\beta\in\R: \beta>1\}$. For simplicity, in this paper, we will also consider the parameter space $(1,2)$.  For all $0\leq a\leq b\leq1$, let
\begin{equation}\label{ab2}
E_{a,b}^P=\left\{\beta\in(1,2):\liminf_{n\rightarrow \infty}\frac{r_n(\beta)}{n}=a,\ \limsup_{n\rightarrow \infty}\frac{r_n(\beta)}{n}=b\right\}.
\end{equation}We have the following theorem.
\begin{theorem}\label{ab'}
The set $E_{0,0}^P$ has full Lebesgue measure.  If $\frac{b}{1+b}< a\leq 1,\ 0< b \leq 1$, then $E_{a,b}^P=\emptyset$. Otherwise, we have  $$\dim_{\rm{H}}E_{a,b}^P=1-\frac{b^2(1-a)}{b-a}.$$
\end{theorem}

Similarly, for every $0\leq a\leq 1$ and $0\leq b\leq 1$, we consider the set
$$E_a^P=\left\{\beta\in(1,2):\liminf_{n\rightarrow \infty}\frac{r_n(\beta)}{n}=a\right\},$$and
$$F_b^P=\left\{\beta\in(1,2):\limsup_{n\rightarrow \infty}\frac{r_n(\beta)}{n}=b\right\}.$$
\begin{corollary}\label{b}
(1) When $0\leq a\leq \frac{1}{2}$, we have $$\dim_{\rm{H}}E_a^P={(1-2a)}^2.$$ Otherwise, $E_{a}=\emptyset$.

(2) For every $0\leq b\leq1$, we have $$\dim_{\rm{H}}F_b^P=1-b.$$
\end{corollary}

In addition, similar to Theorem {\ref{res}}, we have the following theorem.
\begin{theorem}\label{res'}
The set $E_{0,1}^P$ is residual in $[1,2]$.
\end{theorem}

We end this introduction by depicting the organization of our paper. In Section 2, we review some standard facts on the $\beta$-expansions without proofs. Theorem \ref{ab} and corollary \ref{a} are proved in Section 3. We give the proof of Theorem \ref{res} in Section 4. Section 5 contains a summary of some classical results of $\beta$-expansion in the parameter space. The proofs of Theorems \ref{ab'} and \ref{res'} are given in Sections 6 and 7 respectively.

\section{Fundamental results of $\beta$-expansion}
Throughout this section, we set up some notations and terminologies on $\beta$-expansions. Meanwhile, we give some basic results on $\beta$-expansion directly. For more properties on $\beta$-expansions, we refer the readers to \cite{BW,AB,P,R}.

Let $\A=\{0,1,\cdots,\lceil\beta\rceil\}$ where $\lceil\xi\rceil$ stands for the smallest integer larger than $\xi$. The definition of $\beta$-expansion gives the fact that every digit $\varepsilon_n(x,\beta)$ lies in the set $\A$. A word $(\varepsilon_1,\ldots,\varepsilon_n)\in \A^n$ is called \emph{admissible} with respect to $\beta$ if there exists an $x \in [0, 1)$ such that the $\beta$-expansion of $x$ begins with $(\varepsilon_1,\ldots,\varepsilon_n)$. Similarly, an infinite sequence $(\varepsilon_1,\ldots,\varepsilon_n,\ldots)$ is called \emph{admissible} with respect to $\beta$ if there exists an $x \in [0, 1)$ whose $\beta$-expansion is $(\varepsilon_1,\ldots,\varepsilon_n,\ldots)$. Denote by $\Sigma_\b^n$ the set of all $\beta$-admissible words of length $n$, i.e.,\ $$\Sigma_\b^n=\{(\varepsilon_1,\ldots,\varepsilon_n)\in \A^n: \exists\ x \in [0,1),\ {\rm such\ that\ }\varepsilon_j(x,\b)=\varepsilon_j, {\rm\  for \ all }\ 1\leq j \leq n\}.$$ Denote by $\Sigma_\b^\ast$ the set of all $\beta$-admissible words of finite length, i.e.,\ $\Sigma_\b^\ast=\bigcup\limits_{n=0}^\infty \Sigma_\b^n$. The set of $\beta$-admissible sequences is denoted by $\Sigma_\b$ ,  i.e.,\ $$\Sigma_\b=\{(\varepsilon_1,\varepsilon_2,\ldots)\in \A^\N: \exists\ x \in [0,1),\ {\rm s.t.}\ \varepsilon(x,\beta)=(\varepsilon_1,\varepsilon_2,\ldots)\}.$$

The $\beta$-expansion of the unit $1$ plays an important role in the research of admissible words and admissible sequences. We call $\beta$ a \emph{simple Parry number} if the $\beta$ expansion of $1$ is finite. That is, there exists an integer $m\geq 1$ such that $\varepsilon_m\neq 0$ and $\varepsilon_k(\beta)= 0$ for every $k> m$. In this case, we let $$\varepsilon^\ast(\beta):=(\varepsilon_1^\ast(\beta),\varepsilon_2^\ast(\beta),\ldots)=(\varepsilon_1(\beta),\varepsilon_2(\beta),\ldots, \varepsilon_m(\beta)-1)^\infty,$$ where $\omega^\infty$ is the infinite periodic sequence $(\omega,\omega,\ldots)$. If the $\beta$-expansion of $1$ is not finite, let $\varepsilon^\ast(\beta)=\varepsilon(1,\beta)$. In both cases, we can check that $$1=\frac{\varepsilon_1^\ast(\beta)}{\beta}+\cdots+\frac{\varepsilon_n^\ast(\beta)}{\beta^n}+\cdots.$$  The sequence $\varepsilon^\ast(\beta)$ is therefore called \emph{the infinite $\beta$-expansion of $1$}.

We endow the space $\A^\N$ with the \emph{lexicographical order $<_{\rm{lex}}$}: $$(\omega_1,\omega_2,\ldots)<_{\rm{lex}}(\omega'_1,\omega'_2,\ldots)$$if $\omega_1<\omega'_1$ or there exists an integer $j > 1$, such that, for all $1 \leq k<j$, $\omega_k=\omega'_k$  but $\omega_j<\omega'_j$. The symbol $\leq_{\rm{lex}}$ means $=$ or $<_{\rm{lex}}$. Moreover, for all $n,m \geq 1$, $(\omega_1,\ldots,\omega_n) <_{\rm{lex}}(\omega'_1,\ldots,\omega'_m)$ stands for $(\omega_1,\ldots,\omega_n,0^\infty) <_{\rm{lex}}(\omega'_1,\ldots,\omega'_m,0^\infty)$.

The following theorem due to Parry \cite{P} yields that the $\beta$-dynamical system is totally determined by the infinite $\beta$-expansion of $1$. Let $\omega=(\omega_1,\ldots,\omega_n)\in \A^n$ for all $n\geq 1$. Let $\sigma$ be the shift transformation such that $\sigma\omega=(\omega_2,\ldots,\omega_n).$

\begin{theorem}[Parry \cite{P}]\label{P} Let $\beta >1$.

(1) For every $n\geq 1$, $\omega=(\omega_1,\ldots,\omega_n)\in \Sigma_\beta^n$ if and only if $\sigma ^j\omega \leq_{\rm{lex}} (\varepsilon_1^\ast,\ldots,\varepsilon_{n-j}^\ast)\ for \ all\ 0\leq j < n$.

(2) For all $k\geq 1$, $\sigma ^k\varepsilon(1,\beta) <_{\rm{lex}} \varepsilon(1,\beta).$

(3) For each $1<\beta_1<\beta_2$, it holds that $\varepsilon^\ast(\beta_1)<_{\rm{lex}} \varepsilon^\ast(\beta_2)$. Consequently, for every $n \geq 1$, we have $$\Sigma_{\beta_1}^n \subseteq \Sigma_{\beta_2}^n\quad {\rm and}\quad \Sigma_{\beta_1} \subseteq \Sigma_{\beta_2}.$$
\end{theorem}

The estimation of the cardinality of the set $\Sigma_\beta^n$ was given by R\'{e}nyi \cite{R}. We will use the symbol $\sharp$ to denote the cardinality of a finite set in the remainder of this paper.
\begin{theorem}[R\'{e}nyi \cite{R}]\label{R}
For all $n \geq 1$,$$\beta^n \leq \sharp \Sigma_\beta^n \leq \frac{\beta^{n+1}}{\beta-1}.$$
\end{theorem}

For an admissible word $\omega=(\omega_1,\ldots,\omega_n)$, the associated \emph{cylinder} of order $n$ is defined by $$I_n(\omega):=I_n(\omega,\beta)= \{x \in [0,1): \varepsilon_j(x,\b)=\omega_j,\ {\rm for\ all}\  1 \leq j \leq n\}.$$  The cylinder $I_n(\omega)$ is a left-closed and right-open interval (see \cite[Lemma 2.3]{AB}). Denote by  $|I_n(\omega)|$ the length of $I_n(\omega)$. We immediately get $|I_n(\omega)|\leq \beta^{-n}$. We write $I_n(x,\beta)$ as the cylinder of order $n$ containing the point $x\in [0,1)$ and write $|I_n(x,\beta)|$ as its length. For simplicity, $I_n(x)$ means $I_n(x,\beta)$ in the rest of this paper without otherwise specified. A cylinder of order $n$ is called \emph{full} if  $|I_n(\omega)| = \b^{-n}$ and the corresponding word of the full cylinder is said to be {\it full}.

Now we give some characterizations and properties of full cylinders.
\begin{theorem}[Fan and Wang \cite{AB}]\label{AB}
For any integer $n\geq 1$, let $\omega=(\omega_1,\ldots,\omega_n)$ be an admissible word.

(1)The cylinder $I_n(\omega)$ is full if and only if $T_\beta^n(I_n(\omega))=[0,1)$, if and only if  for any $m \geq 1$ and $\omega'= (\omega'_1,\ldots,\omega'_m)\in \Sigma_\b^m$, the concatenation $\omega \ast \omega'=(\omega_1, \ldots,\omega_n,\omega'_1,\ldots, \omega'_m)$ is still admissible.

(2) If $(\omega_1,\ldots,\omega_{n-1},\omega'_n)$ with $\omega'_n > 0$ is admissible, then the cylinder $I_n(\omega_1,\ldots,\omega_{n-1},\omega_n)$ is full for every  $0 \leq \omega_n < \omega'_n$.

(3) If $I_n(\omega)$ is full, then for any $(\omega'_1,\ldots,\omega'_m)\in \Sigma_\b^m$, we have $$|I_{n+m}(\omega_1,\ldots, \omega_n,\omega'_1,\ldots,\omega'_m)| = \beta^{-n}\cdot |I_m(\omega'_1,\ldots,\omega'_m)|.$$
\end{theorem}

In order to construct full words, we introduce a variable $\Gamma_n$ which is defined as follows. Recall that the infinite $\beta$-expansion of $1$ is $(\varepsilon_1^\ast(\beta), \varepsilon_2^\ast(\beta), \ldots)$. For every integer $n \geq 1$, define $$t_n=t_n(\b) := \max\{k \geq 1:\varepsilon_{n+1}^\ast(\beta) =\cdots= \varepsilon_{n+k}^\ast(\beta)=0\}.$$ If such $k$ does not exist, let $t_n=0$. Now let
\begin{equation}\label{g}
\G_n = \G_n(\b) := \max_{1 \leq k \leq n}t_k(\b).
\end{equation} Then we can check that $\G_n$ is a finite integer for all $n\geq 1$. Theorem \ref{AB} implies the following results which are important for construction of full words.
\begin{proposition}[Fan and Wang \cite{AB}]\label{re1}
(1) If both admissible words $(\omega_1,\ldots,\omega_n)$ and $(\omega'_1,\ldots,\omega'_m)$ are full, then the concatenation word $(\omega_1,\ldots,\omega_n,\omega'_1,\ldots,\omega'_m)$ is still full.

(2)For all $\ell\geq 1$, the word $0^\ell:=(\underbrace{0,\ldots,0}_\ell)$ is full. For any full word $(\omega_1,\ldots,\omega_n)$, the word $(\omega_1,\ldots,\omega_n,0^\ell)$ is also full.

(3) For any admissible word $(\omega_1,\ldots,\omega_n)$, the word $(\omega_1,\ldots,\omega_n,0^{\Gamma_n+1})$ is full.
\end{proposition}

Furthermore, Bugeaud and Wang \cite{BW} provided the following modified mass distribution principle which is of great importance in estimating the lower bound of the Hausdorff dimension of $E_{a,b}$.
\begin{theorem}[Bugeaud and Wang \cite{BW}]\label{mp}
Let $\mu$ be a Borel measure and $E$ be a Borel measurable set with $\mu(E)>0$. Assume that there exist a constant $c>0$ and an integer $N\geq 1$ such that for all $n\geq N$ and each cylinder $I_n$, the equality $\mu(I_n)\leq c|I_n|^s$ is valid. Then, $\dim_{\rm H}E\geq s.$
\end{theorem}

Now we will introduce some results on Diophantine approximation. We fist give the following exponents of approximation.

Shen and Wang \cite{SW} obtained the following theorem which gives the dimensional results of the set of points with classical Diophantine property.
\begin{theorem}[Shen and Wang \cite{SW}] \label{sw}
Let $\beta>1$. Let $0\leq v\leq +\infty$. Then $$\dim_{\rm H}\{x\in[0,1]:v_{\beta}(x)\geq v\}=\frac{1}{1+v}.$$
\end{theorem}

Bugeaud and Liao \cite{BL} studied the set of points with uniform Diophantine properties and established the theorem as follows.
\begin{theorem}[Bugeaud and Liao \cite{BL}]\label{bl}
Let $\beta>1$. Let $0<\hat{v}<1$ and $v>0$. If $v<\frac{\hat{v}}{1-\hat{v}}$, then the set $$U_{\beta}(\hat{v},v):=\{x\in[0,1]:\hat{v}_{\beta}(x)= \hat{v},\ v_{\beta}(x)=v\}$$ is empty. Otherwise, we have $$\dim_{\rm H}U_{\beta}(\hat{v},v)=\frac{v-(1+v)\hat{v}}{(1+v)(v-\hat{v})}.$$ Moreover, $$\dim_{\rm H}\{x\in[0,1]:\hat{v}_{\beta}(x)=\hat{v}\}=\left(\frac{1-\hat{v}}{1+\hat{v}}\right)^2.$$
\end{theorem}

\section{ Proofs of Theorem \ref{ab} and Corollary \ref{a} }
Notice that for all $\beta>1$, we have $$\left\{x\in [0,1]:\lim\limits_{n\rightarrow \infty}\frac{r_n(x,\beta)}{\log_\beta n}=1\right\}\subseteq E_{0,0}.$$ In \cite{TYZ}, Tong, Yu and Zhao showed that the set $\left\{x\in [0,1]:\lim\limits_{n\rightarrow \infty}\frac{r_n(x,\beta)}{\log_\beta n}=1\right\}$ is of full Lebesgue measure. As a result, the set $E_{0,0}$ has full Lebesgue measure. Hence, we only need to study the case that $0\leq a\leq 1,\ 0< b\leq 1$. Before we give the proof of Theorem \ref{ab}, we uncover the relationship between run-length function and Diophantine approximation.
\subsection{Run-length function and Diophantine approximation}
\begin{lemma}\label{ship}
Let $\beta>1$. For all $x\in[0,1]$, for any $0<a<1$, we have $\liminf\limits_{n\rightarrow \infty}\frac{r_n(x,\beta)}{n}=a$ if and only if $\hat{v}_{\beta}(x)=\frac{a}{1-a}$.
\end{lemma}
\begin{proof}
$\Rightarrow$) Assume that  $\liminf\limits_{n\rightarrow \infty}\frac{r_n(x,\beta)}{n}=a$, we will give our proof by contradiction.

On the one hand, suppose that $\hat{v}_{\beta}(x)<\frac{a}{1-a}$, then we have $v_0=\frac{a}{2(1-a)}+\frac{\hat{v}_{\beta}(x)}{2}>\hat{v}_{\beta}(x)$. By the definition of $\hat {v}_{\beta}(x)$, there is a sequence $\{n_k\}_{k=1}^\infty$ such that, for all $1\leq n\leq n_k$, $$T_\beta^{n}x>\beta^{-v_0n}\geq \beta^{-(\lfloor v_0n\rfloor+1)}.$$ So it holds that $$r_{n_k+\lfloor v_0n_k\rfloor}(x,\beta)<\lfloor v_0n_k\rfloor+1.$$ This implies that $$\liminf_{n\rightarrow \infty}\frac{r_n(x,\beta)}{n}\leq \lim_{k\rightarrow\infty}\frac{r_{n_k+\lfloor v_0n_k\rfloor}(x,\beta)}{n_k+\lfloor v_0n_k\rfloor}\leq\lim_{k\rightarrow\infty} \frac{\lfloor v_0n_k\rfloor+1}{n_k+\lfloor v_0n_k\rfloor}=\frac{v_0}{1+v_0}=\frac{2a+\hat{v}_{\beta}(x)(1-a)-a}{2+\hat{v}_{\beta}(x)(1-a)-a}<a,$$ where the last inequality follows from
$$\frac{a-x}{b-x}<\frac{a}{b}, {\rm\ for\ all\ } 0\leq a<b,\ x>0.$$ A contradiction with $\liminf\limits_{n\rightarrow \infty}\frac{r_n(x,\beta)}{n}=a$. So $\hat{v}_{\beta}(x)\geq\frac{a}{1-a}$.

On the other hand, suppose that $\hat{v}_{\beta}(x)>\frac{a}{1-a}$, then $v_0=\frac{a}{2(1-a)}+\frac{\hat{v}_{\beta}(x)}{2}<\hat{v}_{\beta}(x)$. The definition of $\hat {v}_{\beta}(x)$ implies that for all $N\gg1$, there exists $1 \leq n \leq N$, such that $$T_\beta^{n}x\leq\beta^{-v_0N}.$$
Then for all $k=N+\lfloor v_0N\rfloor+1\gg 1$, we have $$r_k(x,\beta)\geq \lfloor v_0N\rfloor.$$ This implies that $$\liminf_{k\rightarrow \infty}\frac{r_k(x,\beta)}{k} \geq\lim_{N\rightarrow\infty}\frac{\lfloor v_0N\rfloor}{N+\lfloor v_0N\rfloor+1}= \frac{v_0}{1+v_0}= \frac{2a+\hat{v}_{\beta}(x)(1-a)-a}{2+\hat{v}_{\beta}(x)(1-a)-a}>a,$$ where the last inequality follows from \begin{equation}\label{f}
\frac{a+x}{b+x}> \frac{a}{b}\ {\rm for\ all\ } 0\leq a<b,\ x>0.
\end{equation} This contradicts with $\liminf\limits_{n\rightarrow \infty}\frac{r_n(x,\beta)}{n}=a$. Consequently, $\hat{v}_{\beta}(x)\geq\frac{a}{1-a}$.

\bigskip
$\Leftarrow$) On the one side, if $\hat{v}_{\beta}(x)=\frac{a}{1-a}$,  by the definition of $\hat{v}_{\beta}(x)$, for all $0<\delta<\frac{a}{2(1-a)}$, let $v_1=\frac{a}{1-a}-\delta$. Then for every $N\gg1$, there exists $1 \leq n \leq N$, such that $$T_\beta^{n}x\leq\beta^{-v_1N}.$$  Then for all $k=N+\lfloor v_1N\rfloor+1\gg 1$, we have $$r_k(x,\beta)\geq \lfloor v_1N\rfloor.$$ Then  $$\liminf_{k\rightarrow \infty}\frac{r_k(x,\beta)}{k}\geq \lim_{N\rightarrow\infty} \frac{\lfloor v_1N\rfloor+1}{N+\lfloor v_1N\rfloor}=\frac{v_1}{1+v_1}=\frac{a-(1-a)\delta}{1+(1-a)\delta}.$$ Letting $\delta\rightarrow0$, we have $\liminf\limits_{n\rightarrow \infty}\frac{r_n(x,\beta)}{n}\geq a.$

On the other side, the definition of $\hat{v}_{\beta}(x)$ implies that, for every $v_1'=\frac{a}{1-a}+\delta$ with $\delta>0$, there exists a sequence $\{n_k\}_{k=1}^\infty$ such that, for every $1\leq n\leq n_k$, it holds that $$T_\beta^{n}x>\beta^{-v_1'n}\geq \beta^{-(\lfloor v_1'n\rfloor+1)}.$$ This means $$r_{n_k+\lfloor v_1'n_k\rfloor}(x,\beta)<\lfloor v_1'n_k\rfloor+1.$$ Therefore, $$\liminf_{n\rightarrow \infty}\frac{r_n(x,\beta)}{n}\leq \lim_{k\rightarrow\infty}\frac{r_{n_k+\lfloor v_1'n_k\rfloor}(x,\beta)}{n_k+\lfloor v_1'n_k\rfloor}\leq \lim_{k\rightarrow\infty} \frac{\lfloor v_1'n_k\rfloor+1}{n_k+\lfloor v_1'n_k\rfloor}=\frac{v_1'}{1+v_1'}=\frac{a+(1-a)\delta}{1+(1-a)\delta}.$$ Similarly, by letting $\delta\rightarrow0$, we obtain $\liminf\limits_{n\rightarrow \infty}\frac{r_n(x,\beta)}{n}\leq a.$

Thus, we conclude that $\liminf\limits_{n\rightarrow \infty}\frac{r_n(x,\beta)}{n}= a.$
\end{proof}
\begin{lemma}\label{ship1}
Let $\beta>1$. For all $x\in[0,1)$, for each $0<b<1$, we have $\limsup\limits_{n\rightarrow \infty}\frac{r_n(x,\beta)}{n}=b$ if and only if $v_\beta(x)=\frac{b}{1-b}$.
\end{lemma}
\begin{proof} It can be deduced by the same arguments as the proof of Lemma \ref{ship}.
\end{proof}

Now we can give part of the proof of Theorem \ref{ab}.

We will first show when $a>\frac{b}{1+b},\ 0< b\leq1$, $E_{a,b}=\emptyset$.  In fact, if $\limsup\limits_{n\rightarrow \infty}\frac{r_n(x,\beta)}{n}=b$, then for all $\delta>0$, there exits a sequence $\{n_k\}_{k=1}^\infty$ such that $r_{n_k}(x,\beta)\leq \lfloor (b+\delta)n_k\rfloor$. Thus, when we consider the prefix at the position $n_k+\lfloor bn_k\rfloor$, there are at most $\lfloor(b+\delta)n_k\rfloor$ consecutive $0$'s. Thus  $r_{n_k+\lfloor bn_k\rfloor}(x,\beta)\leq \lfloor (b+\delta)n_k\rfloor$. Hence, $$a=\liminf_{n\rightarrow\infty}\frac{r_n(x,\beta)}{n}\leq \lim_{k\rightarrow \infty} \frac{\lfloor r_{n_k+\lfloor bn_k\rfloor}(x,\beta)\rfloor}{n_k+\lfloor bn_k\rfloor}\leq \lim_{k\rightarrow \infty} \frac{(b+\delta)n_k}{n_k+\lfloor bn_k\rfloor}=\frac{b+\delta}{1+b}.$$ Letting $\delta\rightarrow 0$, we have
\begin{equation}\label{le}
a\leq\frac{b}{1+b}
\end{equation} Therefore, $E_{a,b}$ is empty when $a>\frac{b}{1+b},\ 0< b\leq1$.

When $0<a\leq \frac{b}{1+b},\ 0< b<1$, Lemmas \ref{ship} and \ref{ship1} give the fact that the sets we consider here are essentially the same as the sets studied in Bugeaud and Liao \cite{BL}, that is $$E_{a,b}=\left\{x\in[0,1]:\hat{v}_{\beta}(x)=\frac{a}{1-a},\ v_{\beta}(x)=\frac{b}{1-b}\right\}=U_{\beta}\left(\frac{a}{1-a},\frac{b}{1-b}\right).$$ Consequently, we can apply Theorem \ref{bl} to obtain $$\dim_{\rm H}E_{a,b}=\dim_{\rm H}U_{\beta}\left(\frac{a}{1-a},\frac{b}{1-b}\right)=1-\frac{b^2(1-a)}{b-a}.$$

However, Theorem \ref{bl} cannot be applied for the cases $a=0,\ 0< b<1$ and $0<a\leq\frac{1}{2},\ b=1$. Remark that $E_{a,1}\subseteq F_1$ where $F_1$ is defined by (\ref{bb}) and $\dim_{\rm H}F_1=0$ by  \cite[Theorem 1.1]{LL}. So $\dim_{\rm H} E_{a,1}=0$ ($0<a\leq\frac{1}{2}$) and there is nothing to prove.  For the other case, we have $$E_{0,b}\subseteq \left\{x\in[0,1]:v_{\beta}(x)\geq \frac{b}{1-b}\right\}.$$ Then we can use Theorem \ref{sw} to obtain the upper bound of $\dim_{\rm H}E_{0,b}$ which is $1-b$ for all $0<b<1$. Hence it remains to give the lower bound of $\dim_{\rm H}E_{0,b}$ for all $0<b<1$.

\subsection{Lower bound of $\dim_{\rm H}E_{0,b}\ (0<b<1)$}
Now we give the lower bound of $\dim_{\rm H}E_{0,b}$ for the case $0<b<1$. In fact, we can also include the proof of the lower bound of $\dim_{\rm H}E_{a,b}$ for the case $0<a\leq\frac{b}{1+b},\ 0<b<1,$ though the later case has already been given in the end of Section 3.1.

Let $\beta>1$.  Recall that the infinite $\beta$-expansion of 1 is $\varepsilon^\ast(\beta)=(\varepsilon_1^\ast(\beta), \varepsilon_2^\ast(\beta), \ldots)$. We will apply the approximation of $\beta$ to construct the Cantor subset as follows. For all $N$ with $\varepsilon_N^\ast>0$, let $\beta_N$ be the unique solution of the equation: $$1=\frac{\varepsilon_1^\ast(\beta)}{x}+\cdots+\frac{\varepsilon_N^\ast(\beta)}{x^N}.$$ Then $$\varepsilon^\ast(\beta_N)=(\varepsilon_1^\ast(\beta),\ldots,\varepsilon_{N}^\ast(\beta)-1)^\infty.$$
Hence $0<\beta_N<\beta$ and $\beta_N$ is increasing to $\beta$ as $N$ goes to infinity. The number $\beta_N$ is called \emph{an approximation of $\beta$}.  Moreover, by Theorem \ref{P}(3), $\Sigma_{\beta_N}^n\subseteq\Sigma_{\beta}^n$ for all $n\geq 1$ and $\Sigma_{\beta_N}\subseteq\Sigma_{\beta}$. We therefore have the following facts.
\begin{proposition}[Shen and Wang \cite{SW}]\label{re2}
(1) For all $\omega\in \Sigma_{\beta_N}^n$ with $n\geq N$, the cylinder $I_n(\omega,\beta)$ is full when considering $\omega$ as an element of $\Sigma_{\beta}^n$. Consequently, $\omega$ can concatenate with all $\beta$-admissible words.

(2) For every $\omega\in\Sigma_{\beta_N}^n$, when regarding $\omega$ as an element of $\Sigma_{\beta}^n$, we have
\begin{equation}\label{in}
\beta^{-(n+N)}\leq |I_n(\omega,\beta)|\leq \beta^{-n}.
\end{equation}
\end{proposition}

For all $k\geq 1$ and $N>1$ with $\varepsilon_N^\ast(\beta)>0$, choose two sequences $\{n_k\}_{k=1}^\infty$ and  $\{m_k\}_{k=1}^\infty$ which satisfy  $n_k<m_k<n_{k+1}$ with $n_1>2N$, and $m_k-n_k> m_{k-1}-n_{k-1}$ with $m_1-n_1> 2N$. Moreover, $\{n_k\}_{k=1}^\infty$ and  $\{m_k\}_{k=1}^\infty$ can be chosen to satisfy
\begin{equation}\label{inf2}
\lim_{k\rightarrow \infty}\frac{m_k-n_k}{n_{k+1}+m_k-n_k}=a
\end{equation}
and
\begin{equation}\label{sup2}
\lim_{k\rightarrow \infty}\frac{m_k-n_k}{m_k}=b.
\end{equation}
In fact, such sequences do exist by the following arguments.


(1) If $0<a\leq \frac{b}{1+b},\ 0< b<1$, let
$$n'_k=\left\lfloor\left(\frac{b(1-a)}{a(1-b)}\right)^k\right\rfloor\quad {\rm and}\quad  m'_k=\left\lfloor\frac{1}{1-b}\left(\frac{b(1-a)}{a(1-b)}\right)^k\right\rfloor.$$ Note that $a<b$, so $\frac{b(1-a)}{a(1-b)}>1$. Then both sequences $\{n'_k\}_{k=1}^\infty$ and $\{m'_k\}_{k=1}^\infty$ are increasing to infinity as $k$ tends to infinity. A small adjustment can attain the required sequences.

(2) If $a=0,\ 0< b<1$, let $$n'_k= k^k\quad {\rm and}\quad  m'_k=\left\lfloor\frac{1}{1-b}k^k\right\rfloor.$$ We can adjust these sequences to make sure that $m_k-n_k> m_{k-1}-n_{k-1}$ with $m_1-n_1> 2N$.


Now let us construct a Cantor subset of $E_{a,b}$.

For all $d>2N$, let
\begin{equation}\label{md}
\M_d=\{\omega=(1,0^{N-1},\omega_1,\ldots,\omega_{d-N}): (\omega_1,\ldots,\omega_{d-N}) \in \Sigma_{\beta_N}^{d-N}\}.
\end{equation}Remark that $(1,0^{N-1},\omega_1,\ldots,\omega_{d-N})\in \Sigma_{\beta_N}^d$. Thus, by Proposition \ref{re2}(1), every word belonging to $\M_d$ is full when regarding it as an element of $\Sigma_\beta^d$. Now let
$G_1=\{\omega:\omega\in \M_{n_1}\}$. Next, for all $k\geq 1$, let $n_{k+1}=(m_k-n_k)t_k+m_k+p_k$ where $0\leq p_k<m_k-n_k$. Define
$$G_{k+1}=\{u_{k+1}=(1,0^{m_k-n_k-1},u_k^{(1)},\ldots,u_k^{(t_k)},u_k^{(t_k+1)}):\ u_k^{(i)}\in \M_{m_k-n_k}{\rm \ for\ all\ } 1\leq i\leq t_k \}$$ where
$$u_k^{(t_k+1)}=\left\{
\begin{aligned}
0^{p_k} & , &{\rm when}\ \ p_k \leq 2N; \\
\omega \in \M_{p_k} & , &{\rm when}\ \  p_k > 2N.\\
\end{aligned}
\right.$$

It follows from Propositions \ref{re2}(1) and \ref{re1}(2) that every $u_k\in G_k$ is full. Hence, we can define the set $D_k$ as:
\begin{equation}\label{D}
D_k= \left\{(u_1,\ldots,u_k):u_i\in G_i, \ {\rm for\ all}\ 1\leq i\leq k\right\}.
\end{equation}

Notice that the length of $u_k\in G_k$ satisfies $|u_k|=n_k-n_{k-1}$. For each $u=(u_1,\ldots,u_k)\in D_k,$  we have $$|u|=|u_1|+|u_2|+\cdots+|u_k|=n_1+(n_2-n_1)+\cdots+(n_k-n_{k-1})=n_k.$$ Define $$E_N=\bigcap_{k=1}^{\infty}\bigcup_{u \in D_k}I_{n_k}(u).$$

The following lemma shows that $E_N$ is a subset of $E_{a,b}$.
\begin{lemma}\label{sub}
We have $E_N\subseteq  E_{a,b}$ for every $0\leq a\leq \frac{b}{1+b}$ and $0<b< 1$.
\end{lemma}
\begin{proof} For every integer $n\geq1$, there exists a $k\geq 1$ such that $n_k< n\leq n_{k+1}$. We distinguish three cases.

(1) If $n_k< n\leq n_k+m_{k-1}-n_{k-1}+2N$, we have $m_{k-1}-n_{k-1}-1\leq r_n(x,\beta)\leq m_{k-1}-n_{k-1}+2N$ by the construction of $E_N$. It follows  that$$\frac{m_{k-1}-n_{k-1}-1}{n_k+m_{k-1}-n_{k-1}+2N}\leq \frac{r_n(x,\beta)}{n}\leq \frac{m_{k-1}-n_{k-1}+2N}{n_k}.$$

(2) If $n_k+m_{k-1}-n_{k-1}+2N< n\leq m_k$, the construction of $E_N$ gives $r_n(x,\beta)= n-n_k$. By (\ref{f}), we have $$\frac{m_{k-1}-n_{k-1}+2N}{n_k+m_{k-1}-n_{k-1}+2N}\leq \frac{r_n(x,\beta)}{n}\leq \frac{m_k-n_k}{m_k}.$$

(3) If $m_k\leq n\leq n_{k+1}$, we deduce from the construction of $E_N$ that $m_k-n_k-1\leq r_n(x,\beta)\leq m_k-n_k+2N$. Consequently,$$\frac{m_k-n_k-1}{n_{k+1}}\leq \frac{r_n(x,\beta)}{n}\leq \frac{m_k-n_k+2N}{m_k}.$$

Combining the above three cases, by (\ref{inf2}) and (\ref{sup2}), we have $$\liminf_{n\rightarrow\infty}\frac{r_n(x,\beta)}{n}\geq a\quad {\rm and}\quad \limsup_{n\rightarrow\infty}\frac{r_n(x,\beta)}{n}\leq b.$$

Now we complete our proof by finding the subsequences such that the limit inferior and limit superior are reached. In fact, by (\ref{inf2}), we get $$\lim_{k\rightarrow \infty}\frac{r_{n_k+m_{k-1}-n_{k-1}}}{n_k+m_{k-1}-n_{k-1}}\leq \lim_{k\rightarrow \infty}\frac{m_{k-1}-n_{k-1}+2N}{n_k+m_{k-1}-n_{k-1}}=a.$$ It follows from (\ref{sup2}) that $$\lim_{k\rightarrow \infty}\frac{r_{m_k}}{m_k}= \lim_{k\rightarrow \infty}\frac{m_k-n_k-1}{m_k}=b.$$
\end{proof}

Now we estimate the cardinality of the set $D_k$ defined by (\ref{D}).  Write $q_k:=\sharp D_k.$
\begin{lemma}\label{c}
Let $\beta>1$. Let $\beta_N$ be an approximation of $\beta$. For every $\overline{\beta}<\beta_N$, there exist an integer $k(\overline{\beta}, \beta_N)$ and real numbers $c(\overline{\beta}, \beta_N),\ c'(\overline{\beta}, \beta_N)$  such that, for all $k\geq k(\overline{\beta}, \beta_N)$, we have
\begin{equation}\label{qk}
q_k\geq c'(\overline{\beta}, \beta_N)c(\overline{\beta}, \beta_N)^k\overline{\beta}^{\sum\limits_{i=1}^{k-1}(n_{i+1}-m_i)}.
\end{equation}
\end{lemma}
\begin{proof}
Recall the definition of $\M_d$ as (\ref{md}). Theorem \ref{R} implies $$\sharp \M_d\geq \beta_N^{d-N}$$ for all $d\geq N$. Since $\overline{\beta}<\beta_N$, there exists an integer $d'$ which depends on $\overline{\beta}$ and $\beta_N$ such that, for every $d\geq d'$, we have  \begin{equation}\label{bet}
\beta_N^{d-N}\geq \overline{\beta}^d.
\end{equation}
Moreover, the fact that $m_k-n_k$ is increasing and tends to $+\infty$ as $k\rightarrow+\infty$ ensures that we can find a large enough integer $k(\overline{\beta}, \beta_N)$ satisfying that, for all $k\geq k(\overline{\beta}, \beta_N)$,
\begin{equation}\label{sh}
\sharp \M_{m_k-n_k}\geq \beta_N^{m_k-n_k-N}\geq \overline{\beta}^{m_k-n_k}.
\end{equation}
Then, when $p_k\leq 2N$, we have
$$\sharp G_{k+1}\geq (\sharp \M_{m_k-n_k})^{t_k}\geq  \overline{\beta}^{(m_k-n_k)t_k}\geq \frac{1}{\overline{\beta}^{2N}} \overline{\beta}^{n_{k+1}-m_k}.$$
When $p_k>2N$, we deduce that $$\sharp G_{k+1}\geq (\sharp \M_{m_k-n_k})^{t_k} \cdot \sharp  \M_{p_k} \geq  \overline{\beta}^{(m_k-n_k)t_k}\cdot \beta_N^{p_k-N}= \frac{1}{{\beta_N}^{d'}} \overline{\beta}^{(m_k-n_k)t_k}{\beta_N}^{p_k-N+d'}.$$
Note that $p_k-N+d'>d'$. By (\ref{bet}), we have
$$\sharp G_{k+1}\geq \frac{1}{{\beta_N}^{d'}} \overline{\beta}^{(m_k-n_k)t_k}{\overline{\beta}}^{p_k+d'}= \frac{\overline{\beta}^{d'}}{{\beta_N}^{d'}} \overline{\beta}^{n_{k+1}-m_k}.$$

Let $ c(\overline{\beta},\beta_N):= \min\{\frac{1}{\overline{\beta}^{2N}}, \frac{\overline{\beta}^{d'}}{{\beta_N}^{d'}} \}$. It follows that for all $k\geq k(\overline{\beta}, \beta_N)$, $$\sharp G_{k+1} \geq c(\overline{\beta}, \beta_N) \overline{\beta}^{n_{k+1}-m_k}.$$

Immediately, by the relationship between $D_k$ and $G_k$, for any $k\geq k(\overline{\beta}, \beta_N)$, it comes to the conclusion that
$$\begin{aligned}
q_k=\sharp D_k=\prod_{i=1}^{k}\sharp G_k\geq \prod_{i=k(\overline{\beta},\beta_N)}^k\sharp G_i &\geq c(\overline{\beta}, \beta_N)^{k-k(\overline{\beta},\beta_N)}\overline{\beta}^{\sum\limits_{i=k(\overline{\beta},\beta_N)}^{k-1}(n_{i+1}-m_i)}\\ &\geq c'(\overline{\beta}, \beta_N)c(\overline{\beta}, \beta_N)^k\overline{\beta}^{\sum\limits_{i=1}^{k-1}(n_{i+1}-m_i)},
\end{aligned}$$
where $$c'(\overline{\beta}, \beta_N)=\overline{\beta}^{-\sum\limits_{i=1}^{k(\overline{\beta},\beta_N)-1}(n_{i+1}-m_i)}.$$
\end{proof}

Now we divide into three parts to complete our proof of the lower bound of $\dim_{\rm H}E_{a,b}$ by using the modified mass distribution principle (Theorem \ref{mp}).

(1) Define a probability measure $\mu$ supported on $E_N$. Set $$\mu([0,1])=1\quad {\rm and}\quad \mu(I_{n_1}(u))=\frac{1}{\sharp G_1},\ {\rm for}\ u \in D_1.$$ For each $k\geq 1,$ and $u=(u_1,\ldots,u_{k+1})\in D_{k+1}$, let
\begin{equation}\label{mu}
\mu(I_{n_{k+1}}(u))=\frac{\mu(I_{n_k}(u_1,\ldots,u_k))}{\sharp G_{k+1}}.
\end{equation}  For any $u \notin D_k$ ($k\geq 1$), let $\mu(I_{n_k}(u))=0$. It is routine to check that $\mu$ is well defined on $E_N$ and it can be extended to a probability measure on $[0,1]$.

(2) Calculate the local dimension $\liminf\limits_{n\rightarrow \infty}\frac{\log \mu(I_n)}{\log|I_n|}$ for any $x\in E_N$. For convenience, we denote $I_n(x)$ by $I_n$ without ambiguity. Then we have
\begin{equation}\label{m1}
\mu(I_{n_i})=\frac{1}{q_i}\leq \frac{1}{c'(\beta_N, \overline{\beta})c(\beta_N, \overline{\beta})^i \overline{\beta}^{\sum\limits_{j=1}^{i-1}(n_{j+1}-m_j)}}
\end{equation}for every $i>k(\beta_N, \overline{\beta}),$ where $k(\beta_N, \overline{\beta})$ is an integer given in Lemma \ref{c}. For all $n\geq 1$, there is an integer $k\geq 1$ such that $n_k<n\leq n_{k+1}$. By the construction of $E_N$ and the definition of $\mu$, it is natural to estimate the lower bound of $\frac{\log \mu(I_n)}{\log|I_n|}$ by dividing into the following three cases.

\medskip

Case 1. $n_k<n\leq m_k$. It follows from (\ref{m1}) that $$\mu(I_n)= \mu(I_{n_k})\leq c'(\beta_N, \overline{\beta})^{-1}c(\beta_N, \overline{\beta})^{-k}\overline{\beta}^{-\sum\limits_{j=1}^{k-1}(n_{j+1}-m_j)}.$$ Furthermore, Theorem \ref{AB}(3) implies $$|I_n(x)|\geq |I_{m_k}(x)|=\frac{1}{\beta^{m_k}}.$$
As a consequence, $$\frac{\log \mu(I_n)}{\log|I_n|}\geq \frac{\sum\limits_{j=1}^{k-1}(n_{j+1}-m_j) \log \overline{\beta}+k\log c(\beta_N, \overline{\beta})+\log c'(\beta_N, \overline{\beta})}{m_k \log \beta}.$$

Case 2. $n=m_k+i(m_k-n_k)+\ell$ for some $0\leq i < t_k$ and $0\leq \ell< m_k-n_k$. In this case, when $0\leq \ell\leq N$, by (\ref{m1}) and  (\ref{sh}), we have
$$\begin{aligned}
\mu(I_n)= \mu(I_{m_k+i(m_k-n_k)+\ell})&\leq \mu(I_{m_k})\cdot\frac{1}{(\sharp \M_{m_k-n_k})^i}\\ &\leq c'(\beta_N, \overline{\beta})^{-1}c(\beta_N, \overline{\beta})^{-k}\overline{\beta}^{-\left(\sum\limits_{j=1}^{k-1}(n_{j+1}-m_j)+i(m_k-n_k)\right)}.
\end{aligned}$$
When $N< \ell< m_k-n_k$, we similarly see that
$$\begin{aligned}
\mu(I_n)= \mu(I_{m_k+i(m_k-n_k)+\ell})&\leq \mu(I_{m_k})\cdot\frac{1}{(\sharp \M_{m_k-n_k})^i}\cdot\frac{1}{\Sigma_{\beta_N}^{\ell-N}}\\
&\leq c'(\beta_N, \overline{\beta})^{-1}c(\beta_N, \overline{\beta})^{-k}\overline{\beta}^{-\left(\sum\limits_{j=1}^{k-1}(n_{j+1}-m_j)+i(m_k-n_k)\right)}{\beta_N}^{-\ell+N}.
\end{aligned}$$
Moreover, by (\ref{in}), it holds that  $$|I_n|\geq |I_{m_k+i(m_k-n_k)+\ell}|\geq \frac{1}{\beta^{m_k+i(m_k-n_k)+\ell+N}}.$$

Therefore,$$\frac{\log \mu(I_n)}{\log|I_n|}\geq \frac{\left(\sum\limits_{j=1}^{k-1}(n_{j+1}-m_j)+i(m_k-n_k)\right) \log \overline{\beta}+(\ell-N)\log\beta_N+k\log c(\beta_N, \overline{\beta})+\log c'(\beta_N, \overline{\beta})}{(m_k+i(m_k-n_k)+\ell+N) \log \beta}.$$

Case 3. $n=m_k+t_k(m_k-n_k)+\ell$ where $0\leq\ell\leq p_k$.  When $0\leq\ell\leq 2N$, we have
$$\begin{aligned}
\mu(I_n)= \mu(I_{m_k+t_k(m_k-n_k)})&= \mu(I_{m_k}(x))\cdot\frac{1}{(\sharp \M_{m_k-n_k})^{t_k}}\\&\leq c'(\beta_N, \overline{\beta})^{-1}c(\beta_N, \overline{\beta})^{-k}\overline{\beta}^{-\left(\sum\limits_{j=1}^{k-1}(n_{j+1}-m_j)+t_k(m_k-n_k)\right)}.\end{aligned}$$When $2N<\ell\leq p_k$, we get
$$\begin{aligned}
\mu(I_n)= \mu(I_{m_k+t_k(m_k-n_k)+\ell})&\leq \mu(I_{m_k}(x))\cdot\frac{1}{(\sharp\ M_{m_k-n_k})^{t_k}}\cdot\frac{1}{\Sigma_{\beta_N}^{\ell-N}}\\
&\leq c'(\beta_N, \overline{\beta})^{-1}c(\beta_N, \overline{\beta})^{-k}\overline{\beta}^{-\left(\sum\limits_{j=1}^{k-1}(n_{j+1}-m_j)+t_k(m_k-n_k)\right)}{\beta_N}^{-\ell+N}.
\end{aligned}$$
In addition, by (\ref{in}),  $$|I_n|\geq |I_{m_k+t_k(m_k-n_k)+\ell}|\geq\frac{1}{\beta^{m_k+t_k(m_k-n_k)+\ell+N}}.$$
Hence,  $$\frac{\log \mu(I_n)}{\log|I_n|}\geq \frac{\left(\sum\limits_{j=1}^{k-1}(n_{j+1}-m_j)+t_k(m_k-n_k)\right) \log \overline{\beta}+(\ell-N)\log\beta_N+k\log c(\beta_N, \overline{\beta})+\log c'(\beta_N, \overline{\beta})}{(m_k+t_k(m_k-n_k)+\ell+N) \log \beta}.$$

In all three cases, using (\ref{f}), we obtain$$\liminf_{n\rightarrow\infty}\frac{\log \mu(I_n)}{\log|I_n|}\geq \lim_{k\rightarrow\infty}\frac{\sum\limits_{j=1}^{k-1}(n_{j+1}-m_j)}{m_k}\frac{\log\overline{\beta}}{\log\beta}.$$ By (\ref{inf2}) and (\ref{sup2}), it immediately holds that $$\lim_{k\rightarrow\infty}\frac{n_k}{m_k}=1-b,\quad \lim_{k\rightarrow\infty}\frac{n_{k+1}}{m_k}=\frac{b(1-a)}{a}\quad {\rm and}\quad \lim_{k\rightarrow\infty}\frac{m_{k+1}}{m_k}=\frac{b(1-a)}{a(1-b)}.$$ By the Stolz-Ces\`{a}ro Theorem, we have$$
\lim_{k\rightarrow\infty}\frac{\sum\limits_{j=1}^{k-1}(n_{j+1}-m_j)}{m_k}=\lim_{k\rightarrow\infty}\frac{n_{k+1}-m_k}{m_{k+1}-m_k}
=\lim_{k\rightarrow\infty}\frac{\frac{n_{k+1}}{m_k}-1}{\frac{m_{k+1}}{m_k}-1}=1-\frac{b^2(1-a)}{b-a}.$$
As a consequence, $$\liminf_{n\rightarrow \infty}\frac{\log \mu(I_n)}{\log|I_n|}\geq \left(1-\frac{b^2(1-a)}{b-a}\right)\frac{\log\overline{\beta}}{\log\beta}.$$

(3) Use the modified mass distribution principle (Theorem \ref{mp}). We first let $\overline{\beta}\rightarrow\beta_N$, and then let $N\rightarrow\infty$. Applying Theorem \ref{mp}, we finish our proof.

\subsection{Proof of Corollary \ref{a}}
Note that when $\frac{1}{2}< a\leq 1$, the inequality (\ref{le}) implies $E_a=\emptyset$. We only need to consider the case $0\leq a\leq \frac{1}{2}$. By Lemma \ref{ship}, we have $$E_a=\left\{x\in[0,1]:\hat{v}_{\beta}(x)=\frac{a}{1-a}\right\}.$$  Thus, applying Theorem \ref{bl}, we have, for all $0<a\leq\frac{1}{2}$, $$\dim_{\rm H} E_a=\dim_{\rm H}\left\{x\in[0,1]:\hat{v}_{\beta}(x)=\frac{a}{1-a}\right\}=(1-2a)^2.$$

When $a=0$, by noting that $E_{0,0}\subseteq E_0$, we deduce that $E_0$ has full Lebesgue measure and thus has Huasdorff dimension $1$.
\section{Proof of Theorem \ref{res}}
The key point to prove Theorem \ref{res} is constructing a set $U$  with the following properties: (1) $U$ is a subset of  $E_{0,1}$; (2) $U$ is dense in the interval $[0,1]$; (3) $U$ is a $G_\d$ set, i.e.,\ a countable intersection of open sets.

Let $\beta>1$. Define $M=\min\{i>1:\varepsilon^\ast_i(\beta)>0\}.$ For all $k\geq 1$, let $\Gamma_k$ be defined by (\ref{g}). We choose two sequences $\{n_k\}_{k=1}^\infty$ and $\{m_k\}_{k=1}^\infty$ such that $n_k<m_k<n_{k+1}$ with $n_k>2k+\Gamma_k$ and $m_k-n_k>\max\{2(m_{k-1}-n_{k-1}),n_k-k,M\}$. In addition, $\{n_k\}_{k=1}^\infty$ and $\{m_k\}_{k=1}^\infty$ satisfy $$\lim_{k\rightarrow \infty}\frac{m_k-n_k}{n_{k+1}+m_k-n_k}=0,$$and$$\lim_{k\rightarrow \infty}\frac{m_k-n_k}{m_k}=1.$$
In fact, let $$n'_k= (2k+\Gamma_k)^{2k}\quad {\rm and}\quad  m'_k=(2k+2+\Gamma_{k+1})^{2k+1}.$$ Then by small adjustments, we can obtain the required sequences.

For all $k\geq 1$, write $n_{k+1}=(m_k-n_k)t_k+n_k+p_k$ where $0\leq p_k <m_k-n_k$.
Now we define $$U:=\bigcap_{n=1}^{\infty} \bigcup_{k=n}^{\infty} \bigcup_{(\epsilon_1,\ldots,\epsilon_k) \in \Sigma_{\beta}^k}{\rm int} \left(I_{n_{k+1}}(\epsilon_1,\ldots,\epsilon_k ,0^{n_k-k},(1,0^{m_k-n_k-1})^{t_k},0^{p_k})\right),$$ where ${\rm int}(I_{|\epsilon|}(\epsilon))$ stands for the interior of $I_{|\epsilon|}(\epsilon)$ for all $\epsilon \in \Sigma^\ast_\beta$.

\begin{remark}
For all $(\epsilon_1,\ldots,\epsilon_k)\in \Sigma_{\beta}^k$, it follows from Proposition \ref{re1}(3) that $(\epsilon_1,\ldots,\epsilon_k,0^{n_k-k})$ is full since $n_k>2k+\Gamma_k$. Note that $m_k-n_k\geq M$. Then the word $(1,0^{m_k-n_k-1})$ is full. By Proposition \ref{re1}(2), the word $0^{p_k}$ is full. Thus $U$ is well defined.
\end{remark}

The set ${\rm int}(I_{|\epsilon|}(\epsilon))$ is open which implies that $U$ is a $G_{\d}$ set. Consequently, it suffices to show that $U$ is a subset of $E_{0,1}$ and is dense in $[0,1]$.
\begin{lemma}\label{re}
The set $U$ is a subset of $E_{0,1}$.
\end{lemma}
\begin{proof}
For any $x \in U,$ it follows from the construction of $U$ that there exist infinitely many $k$ such that $\varepsilon(x,\beta)=(\epsilon_1,\ldots,\epsilon_k ,0^{n_k-k},(1,0^{m_k-n_k-1})^{t_k},0^{p_k})$ for some $(\epsilon_1,\ldots,\epsilon_k)\in \Sigma_\beta^k$. Now we are going to give the upper limit and lower limit of $\frac{r_n(x,\beta)}{n}$.

Let $n=n_{k+1}+m_k-n_k-1$. Since $m_k-n_k>\max\{2(m_{k-1}-n_{k-1}),n_k-k,M\}$, we obtain $$r_{n_{k+1}+m_k-n_k-1}(x,\beta)= m_k-n_k-1.$$ As a result,
  $$\liminf_{n\rightarrow \infty} \frac{r_n(x,\beta)}{n}\leq \lim_{k\rightarrow \infty}\frac{r_{n_{k+1}+m_k-n_k-1}(x,\beta)}{n_{k+1}+m_k-n_k-1}= \lim_{k\rightarrow \infty}\frac{m_k-n_k-1}{n_{k+1}+m_k-n_k-1}=0.$$

Let $n=m_k$. Note that $m_k-n_k>\max\{2(m_{k-1}-n_{k-1}),n_k-k,M\}$. The definition of $r_n(x,\beta)$ shows that $$r_{m_k}(x,\beta)=m_k-n_k-1.$$ It therefore follows that $$\limsup_{n\rightarrow \infty} \frac{r_n(x,\beta)}{n}\geq \lim_{k\rightarrow \infty}\frac{r_{m_k}(x,\beta)}{m_k}= \lim_{k\rightarrow \infty}\frac{m_k-n_k-1}{m_k}= \lim_{k\rightarrow \infty}\frac{m_k-n_k-1}{m_k}=1.$$

By the above discussion, we conclude that $$\liminf_{n\rightarrow \infty} \frac{r_n(x,\beta)}{n}=0\quad {\rm and}\quad \limsup_{n\rightarrow \infty} \frac{r_n(x,\beta)}{n}=1.$$Hence, $x\in E_{0,1}$ which gives $U \subseteq E_{0,1}$.
\end{proof}

\textbf{Proof of Theorem \ref{res}} It remains to show that for all $n\geq 1$, the set $$U_n=\bigcup_{k=n}^\infty\bigcup_{(\epsilon_1,\ldots,\epsilon_k) \in \Sigma_{\beta}^k}{\rm int} \left(I_{n_{k+1}}\left(\epsilon_1,\ldots,\epsilon_k ,0^{n_k-k},(1,0^{m_k-n_k})^{t_k},0^{p_k}\right)\right)$$ is dense in [0,1]. Now we will  concentrate on finding a real number $y\in U$ such that $|x-y|\leq r$ for every  $x\in [0,1]$ and $r > 0$. Suppose that $\varepsilon(x,\b)=(\varepsilon_1(x,\beta),\varepsilon_2(x,\beta),\ldots)$. Let $\ell'$ be an integer satisfying $\b^{-\ell'} \leq r$.  Let $\ell=\max\{n,\ell'\}.$  Since $(\varepsilon_1(x,\beta),\ldots,\varepsilon_\ell(x,\beta)) \in \Sigma^\ell_\beta$, we choose a point $$y\in {\rm int}\left(I_{n_{\ell+1}}\left(\epsilon_1,\ldots,\epsilon_\ell ,0^{n_\ell-\ell},(1,0^{m_\ell-n_\ell})^{t_\ell},0^{p_\ell}\right)\right).$$ Then it holds that $|x-y|\leq \b^{-\ell} \leq r$ and $y \in U_n$. To sum up, the set $$\bigcup_{k=n}^\infty\bigcup_{(\epsilon_1,\ldots,\epsilon_k) \in \Sigma_{\beta}^k}{\rm int} \left(I_{n_{k+1}} \left(\epsilon_1,\ldots,\epsilon_k ,0^{n_k-k},(1,0^{m_k-n_k})^{t_k},0^{p_k}\right)\right)$$ is dense in $[0,1]$.

Thus, we can conclude by the Baire Category Theorem that $U$ is residual in $[0,1]$. Then, $E_{0,1}$ is residual in $[0,1]$ by Lemma \ref{re}.

$\hfill\Box$

\section{Classical results of $\beta$-expansion in the parameter space}
In this section, we recall some important results of $\beta$-expansion in the parameter space $\{\beta\in\R:\beta>1\}$. The readers can refer to \cite{CC,HTY,LP,P,S} for more information.
\begin{definition}
We call a word $\omega=(\omega_1, \ldots,\omega_n)$ self-admissible if for all $1\leq i<n$, $$ \sigma ^i\omega \leq_{\rm{lex}} (\omega_1,\ldots,\omega_{n-i}). $$
An infinite sequence $\omega=(\omega_1, \omega_2,\ldots)$ is called self-admissible if $ \sigma ^i\omega <_{\rm{lex}}\omega$ for all $i\geq 1$.
\end{definition}

Denote by $\Lambda_n$ the set of all self-admissible words with length $n$, i.e.,\ $$\Lambda_n=\{\omega=(\omega_1, \omega_2,\ldots,\omega_n):\ {\rm for\ every}\ 1\leq i<n,\ \sigma ^i\omega \leq_{\rm{lex}} (\omega_1,\ldots,\omega_{n-i}) \}.$$ For convenience, for all $1<\beta_1<\beta_2$, let
\begin{equation}\label{la}
\Lambda_n(\beta_1,\beta_2)=\{\omega=(\omega_1,\ldots,\omega_n)\in \Lambda_n: \exists\ \beta\in(\beta_1,\beta_2]:\ {\rm s.t. }\ \varepsilon_1(\beta)=\omega_1,\ldots,\varepsilon_n(\beta)=\omega_n\}.
\end{equation}

The definition of self-admissible word immediately gives the following result.  The proof is evident and will be omitted.

\begin{proposition}\label{se}
For any $m\geq n\geq1$, let $\omega\in \Lambda_n$. Let $\beta>1$ whose infinite $\beta$-expansion of 1 satisfy $(\varepsilon_1^\ast(\beta),\ldots,\varepsilon_n^\ast(\beta))<_{\rm lex} \omega$. Then for all $v_1,v_2,\ldots,v_i\in \Sigma_\beta^m\ (i\geq 1)$, the concatenation  $\omega\ast v_1 \ast \cdots \ast v_j$ is still self-admissible for all $1\leq j\leq i$.
\end{proposition}

The characterization of the the $\beta$-expansion of $1$ was given by Parry \cite{P}.
\begin{theorem}[Parry \cite{P}]\label{P2}
An infinite sequence $(\omega_1,\omega_2,\ldots)$ is the  $\beta$-expansion of $1$ for some $\beta>1$ if and only if it is self-admissible.
\end{theorem}

Now we consider the \emph{cylinders in the parameter space} $\{\beta\in\R:\beta>1\}$.
\begin{definition}
For any $\omega=(\omega_1,\ldots,\omega_n)\in \Lambda_n$. The cylinder $I_n^P(\omega)$ associated to $\omega$ in the parameter space is the set of $\beta\in(1, +\infty)$ whose $\beta$-expansion of $1$ has the prefix $(\omega_1,\ldots,\omega_n)$, i.e.\  $$I^P_n(\omega):=\{\beta\in(1,+\infty):\varepsilon_1(\beta)=\omega_1,\ldots, \varepsilon_n(\beta)=\omega_n\}.$$
\end{definition}

The cylinders in the parameter space are intervals (see \cite[Lemma 4.1]{S}). The length of the cylinders of $\omega\in \Lambda_n$ in the parameter space is denoted by $|I^P_n(\omega)|$.  For simplicity, the left endpoint and right endpoint of $I^P_n(\omega)$ are written as $\underline{\beta}(\omega)$  and $\overline{\beta}(\omega)$ respectively.

To estimate the length of cylinders in the parameter space, we need the notion of \emph{recurrence time} $\tau(\omega)$ (see \cite{LP}) of the self-admissible word $\omega=(\omega_1,\ldots,\omega_n)\in \Lambda_n$. Define
$$\tau(\omega):=\inf\{1\leq k < n:\sigma^k(\omega_1,\ldots,\omega_n)=(\omega_1,\ldots, \omega_{n-k})\}.$$
If we cannot find such an integer $k$, we set $\tau(\omega)=n$. In this case, the self-admissible word $\omega$ is said to be \emph{non-recurrent}.

The above definition of recurrence time immediately provides the following properties.
\begin{remark}
(1) Write $$t(\omega):= n-\left \lfloor\frac{n}{\tau(\omega)}\right \rfloor\tau(\omega).$$
Then we have
$$(\omega_1,\ldots,\omega_n)=\left((\omega_1,\ldots,\omega_{\tau(\omega)})^{\lfloor \frac{n}{\tau(\omega)}\rfloor},\omega_1,\ldots,\omega_{t(\omega)}\right).$$

(2) If $\omega=(\omega_1,\ldots,\omega_n)$ is non-recurrent, then the word $(\omega_1,\ldots,\omega_n,0^\ell)$ is still non-recurrent for all $\ell \geq 1$.
\end{remark}


The following result gives the upper and lower bounds of the length of the cylinder $I^P_n(\omega)$.
\begin{lemma}[Schemling \cite{S}, Li, Persson, Wang and Wu \cite{LP}]\label{S}
Let $\omega=(\omega_1,\ldots,\omega_n)\in \Lambda_n$. We have the following inequalities:

(1) $|I^P_n(\omega)|\leq \overline{\beta}(\omega)^{-n+1};$

(2)$$|I^P_n(\omega)|\geq \left\{
\begin{aligned}
C(\omega)\overline{\beta}(\omega)^{-n}\ & , & when\ t(\omega)=0; \\
C(\omega)\overline{\beta}(\omega)^{-n}\left(\frac{\omega_{t(\omega)+1}}{\overline{\beta}(\omega)}+ \cdots+\frac{\omega_{\tau(\omega)}+1}{\overline{\beta}(\omega)^{\tau(\omega)-t(\omega)}}\right) & , & otherwise,
\end{aligned}
\right.$$where
\begin{equation}\label{cn}
C(\omega):=\frac{(\underline{\beta}(\omega)-1)^2}{\underline{\beta}(\omega)}.
\end{equation}
\end{lemma}

The study of the parameter space usually concerns on the set of parameters with respect to which the approximation properties of the orbit of $1$ are prescribed. Persson and Schmeling \cite{PS} proved the following result.
\begin{theorem}[Persson and Schmeling \cite{PS}]\label{ps}
Let $v\geq0$. Then $$\dim_{\rm H}\{\beta\in(1,2):v_{\beta}(1)\geq v\}=\frac{1}{1+v}.$$
\end{theorem}

Analogous to Theorem \ref{bl}, Bugeaud and Liao \cite{BL} obtained the following theorem in the parameter space.
\begin{theorem}[Bugeaud and Liao \cite{BL}]\label{bl1}
Let $\beta>1$. Let $0<\hat{v}<1$ and $v>0$. If $v<\frac{\hat{v}}{1-\hat{v}}$, then the set $$U(\hat{v},v):=\{\beta\in(1,2):\hat{v}_{\beta}(1)= \hat{v},\ v_{\beta}(1)=v\}$$ is empty. Otherwise, we have $$\dim_{\rm H}U(\hat{v},v)=\frac{v-(1+v)\hat{v}}{(1+v)(v-\hat{v})}.$$ Moreover, $$\dim_{\rm H}\{\beta\in(1,2):\hat{v}_{\beta}(1)=\hat{v}\}=\left(\frac{1-\hat{v}}{1+\hat{v}}\right)^2.$$
\end{theorem}

\section{ Proof of Theorem \ref{ab'}}
As the same discussion at the first part of Section 3, it holds that $\dim_{\rm H} E^P_{0,0}$ is of full Lebesgue measure by using the result of Cao and Chen \cite{CC} that the set $$\left\{\beta\in(1,2):\lim\limits_{n\rightarrow \infty}\frac{r_n(\beta)}{\log_\beta n}=1\right\}$$ is of full Lebesgue measure. By the same argument as the proof of Theorem \ref{ab} for the case  $a>\frac{b}{1+b},\ 0< b\leq1$ in the end of Section 3.1, we get that $E^P_{a,b}$ is empty when $a>\frac{b}{1+b},\ 0< b\leq1$.

When $0<a\leq \frac{b}{1+b},\ 0< b<1$, Lemmas \ref{ship} and \ref{ship1} imply that $$E^P_{a,b}=\left\{\beta \in(1,2):\hat{v}_{\beta}(1)=\frac{a}{1-a},\ v_{\beta}(1)=\frac{b}{1-b}\right\}=U\left(\frac{a}{1-a},\frac{b}{1-b}\right).$$ Then by Theorem \ref{bl1}, it holds that $$\dim_{\rm H}E^P_{a,b}=\dim_{\rm H}U\left(\frac{a}{1-a},\frac{b}{1-b}\right)=1-\frac{b^2(1-a)}{b-a}.$$

But Theorem \ref{bl1} is not applicable for the case of $a=0,\ 0< b<1$ and $0<a\leq\frac{1}{2},\ b=1$. Note that $E^P_{a,1}\subseteq F^P_1$ where $F^P_1$ is defined by (\ref{bb}). So we first give the Hausdorff dimension of $F^P_1$. Since $$F^P_1=\left\{\beta\in(1,2):v_{\beta}(1)=+\infty\right\}\subseteq  \left\{\beta\in(1,2):v_{\beta}(1)\geq v\right\}$$ for all $v>0$, we have $$\dim_{\rm H}F^P_1\leq \dim_{\rm H}\left\{\beta\in(1,2):v_{\beta}(1)\geq v\right\}=\frac{1}{1+v}$$ where the last equality follows from Theorem \ref{ps}. Letting $v\rightarrow +\infty$, we have $\dim_{\rm H}F^P_1\leq0$. This implies $\dim_{\rm H}E^P_{a,1}\leq0$. In conclusion, $\dim_{\rm H}E^P_{a,1}=0$ for any $0< a\leq \frac{1}{2}$. For the other case, we have $$E^P_{0,b}\subseteq \left\{\beta\in(1,2):v_{\beta}(1)\geq \frac{b}{1-b}\right\}.$$ By Theorem \ref{ps}, we deduce that the upper bound of $\dim_{\rm H}E^P_{0,b}$ is $1-b$ for all $0<b<1$. Hence, we only need to give the lower bound of $\dim_{\rm H}E_{0,b}$ for all $0<b<1$. We also include our proof of the case $0<a\leq \frac{b}{1+b},\ 0< b<1$.

For every $1<\beta_1<\beta_2<2$, instead of dealing with the Hausdorff dimension of the set $E_{a,b}^P$ directly, we will technically investigate the Hausdorff dimension of the following set. For all $0\leq a\leq \frac{b}{1+b},\ 0< b\leq 1$, let
\begin{equation}\label{ab3}
E_{a,b}^P(\beta_1,\beta_2)=\left\{\beta\in[\beta_1,\beta_2):\liminf_{n\rightarrow \infty}\frac{r_n(\beta)}{n}=a,\ \limsup_{n\rightarrow \infty}\frac{r_n(\beta)}{n}=b\right\}.
\end{equation}

For all $1<\beta_1<\beta_2<2$, throughout this section, we assume that both $\beta_1$ and  $\beta_2$ are not simple Parry number. We will give the lower bound of $\dim_{\rm H} E_{a,b}^P(\beta_1,\beta_2)$ for all $1<\beta_1<\beta_2<2$.

Suppose that $N$ is a large enough integer such that $\varepsilon_N(\beta_2)>0$ and $$(\varepsilon_1(\beta_1),\ldots,\varepsilon_N(\beta_1))<_{\rm lex}(\varepsilon_1(\beta_2),\ldots,\varepsilon_N(\beta_2))$$ Let $\widetilde{\beta}_N$ be the unique solution of the  equation:$$1=\frac{\varepsilon_1(\beta_2)}{x}+\cdots+\frac{\varepsilon_N(\beta_2)}{x^N}.$$ Then $$\varepsilon^\ast(\widetilde{\beta}_N)=(\varepsilon_1(\beta_2),\ldots,\varepsilon_{N}(\beta_2)-1)^\infty.$$ An observation of the lexicographical order of $\varepsilon^\ast(\beta_1)$, $\varepsilon^\ast(\beta_2)$ and $\varepsilon^\ast(\widetilde{\beta}_N)$ implies $\beta_1<\widetilde{\beta}_N<\beta_2$ and $\widetilde{\beta}_N\rightarrow \beta_2$ as $N$ tends to infinity.

For every $k\geq 1$, similar to what we did in Section 3.2, we take two sequences $\{n_k\}_{k=1}^\infty$ and $\{m_k\}_{k=1}^\infty$ such that $n_k<m_k<n_{k+1}$ with $n_1>2N$ and $m_k-n_k> m_{k-1}-n_{k-1}$ with $m_1-n_1> 2N$. In addition,
$$\lim_{k\rightarrow \infty}\frac{m_k-n_k}{n_{k+1}+m_k-n_k}=a\quad {\rm and}\quad \lim_{k\rightarrow \infty}\frac{m_k-n_k}{m_k}=b.$$
We can choose such two sequences by the same way in Section 3.2.

Now let us construct a Cantor set contained in $E_{a,b}^P(\beta_1,\beta_2)$ as follows.

For any integer $d>2N$, we set
\begin{equation}
\M'_d=\{\omega=(\varepsilon_1(\beta_2),\ldots,\varepsilon_N(\beta_2)-1,\omega_1,\ldots,\omega_{d-2N},0^N): (\omega_1,\ldots,\omega_{d-2N}) \in \Sigma_{\widetilde{\beta}_N}^{d-2N}\}.
\end{equation} Let
$$G'_1=\{(\varepsilon_1(\beta_2),\ldots,\varepsilon_N(\beta_2),\omega_1,\ldots,\omega_{d-2N},0^N): (\omega_1,\ldots,\omega_{d-2N}) \in \Sigma_{\widetilde{\beta}_N}^{d-2N}\}.$$

Note that $\left(\varepsilon^\ast_1(\widetilde{\beta}_N),\ldots,\varepsilon^\ast_N(\widetilde{\beta}_N)\right)<_{\rm lex}(\varepsilon_1(\beta_2),\ldots,\varepsilon_{N}(\beta_2))$. Now we give some observations on the elements in $G'_1$ as follows.
\begin{remark}
(1) For all $\omega\in G'_1$, since $\left(\omega_1,\ldots,\omega_{d-2N},0^N\right)\in \Sigma_{\widetilde{\beta}_N}^{d-N}\ (d>2N)$, by Proposition \ref{se}, $\omega$ is self-admissible.

(2) For every $u\in \M'_d\ (d>2N)$, we have $u\in \Sigma_{\widetilde{\beta}_N}^d$. By Lemma \ref{se}, the word $\omega\ast u$ is still self-admissible for every  all  $\omega\in G'_1$.
\end{remark}

For every $k\geq 1$, write $n_{k+1}=(m_k-n_k)t_k+m_k+p_k$ with $0\leq p_k<m_k-n_k$, then define
$G'_{k+1}=$\\$$\{u_{k+1}=\left(\varepsilon_1(\beta_2),\ldots,\varepsilon_N(\beta_2)-1,0^{m_k-n_k-N},u_k^{(1)},\ldots,u_k^{(t_k)},u_k^{(t_k+1)}\right):\ u_k^{(i)}\in \M'_{m_k-n_k},\ 1\leq i\leq t_k \},$$ where
$$u_k^{(t_k+1)}=\left\{
\begin{aligned}
0^{p_k} & , &{\rm when}\ \ p_k \leq 2N; \\
\omega \in \M'_{p_k} & , &{\rm when}\ \  p_k > 2N.\\
\end{aligned}
\right.$$
Let
\begin{equation}
D'_k= \left\{(u_1,\ldots,u_k):u_i\in G'_i,\ 1\leq i\leq k\right\}.
\end{equation}
Notice that every $u_k\in G'_k$ ends with $0^N$. This guarantees that $(u_1,\ldots,u_k)$ can concatenate with any $u_{k+1}$ to be a new self-admissible word. As a result, the set $D'_k$ is well-defined.

As the classical technique of constructing a Cantor set, let $$E(\beta_1,\beta_2)=\bigcap_{k=1}^\infty\bigcup_{u \in D'_k}I_{n_k}^P(u).$$

Similar to the process of Section 3, we now give the following result which means that $E(\beta_1,\beta_2)$ is a subset of $E_{a,b}^P(\beta_1,\beta_2)$.
\begin{lemma}
For every $1<\beta_1<\beta_2<2$, $E(\beta_1,\beta_2)\subset E_{a,b}^P(\beta_1,\beta_2)$ for all $0\leq a\leq \frac{b}{1+b}$ and $0<b< 1$.
\end{lemma}
\begin{proof}
The proof is just as the same as the proof of Lemma \ref{sub} by dividing into three cases. We omit it here.
\end{proof}

Analogously, we now focus on the estimation of the cardinality of the set $D'_k$.  Let $q'_k:=\sharp D'_k.$ We obtain the following lemma.
\begin{lemma}\label{c'}
For every $1<\beta_1<\beta_2<2$, let $\widetilde{\beta}_N$ be the real number defined in this section. Then there exist an integer $k(\beta_1, \widetilde{\beta}_N)$ and real numbers $c(\beta_1, \widetilde{\beta}_N),c'(\beta_1, \widetilde{\beta}_N)$  such that, for every $k\geq k(\beta_1, \widetilde{\beta}_N)$, we have
\begin{equation}
q'_k\geq c'(\beta_1, \widetilde{\beta}_N)c(\beta_1, \widetilde{\beta}_N)^k\beta_1^{\sum\limits_{i=1}^{k-1}(n_{i+1}-m_i)}.
\end{equation}
\end{lemma}
\begin{proof}
We use the similar method as Lemma \ref{c}, the details are left to the readers.
\end{proof}

Let $$C(\beta_1)=\frac{(\beta_1-1)^2}{\beta_1}.$$ Notice that $\underline{\beta}(u)\geq \beta_1>1$ for any $u=(u_1,\ldots,u_n)\in \Lambda_n(\beta_1,\beta_2)$ where $\Lambda_n(\beta_1,\beta_2)$ is defined by (\ref{la}). Then $$C(\underline{\beta}(u))= \frac{(\underline{\beta}(u)-1)^2}{\underline{\beta}(u)}\geq C(\beta_1).$$ The following lemma gives the estimation of the length of the cylinders with non-empty intersection with the Cantor set $E(\beta_1,\beta_2)$ which will be useful to estimate the local dimension $\liminf\limits_{n\rightarrow \infty}\frac{\log \mu\left(B(\beta,r)\right)}{\log|r|}$ for any $r>0$ and $\beta\in E(\beta_1,\beta_2)$.

\begin{lemma}\label{i}
For any $\beta\in E(\beta_1,\beta_2)$, suppose $\varepsilon(1,\beta)=(u_1,u_2,\ldots)$. Then we have $$|I_n^P(u_1,\ldots,u_n)|\geq C(\beta_1)\beta_2^{-(n+N)}$$ for any $n\geq 1.$
\end{lemma}
\begin{proof}
For any $n\geq1$, we are going to take the word $(u_1,\ldots,u_n,0^N)$ into account. We claim that the word $(u_1,\ldots,u_n,0^N)$ is non-recurrent.

In fact, by the construction of $E(\beta_1,\beta_2)$, for any $1\leq i<n$, we have $$\sigma^i(u_1,\ldots,u_n,0^N)\in \Sigma_{\widetilde{\beta}_N}^{n-i+N}.$$ Notice that $\omega \leq_{\rm lex} (\varepsilon_1(\beta_2),\ldots,\varepsilon_{N}(\beta_2)-1)<_{\rm lex}(\varepsilon_1(\beta_2),\ldots,\varepsilon_{N}(\beta_2))$ for any $\omega\in\Sigma_{\widetilde{\beta}_N}^{n}$ with $n\geq N$. It comes to the conclusion that $\sigma^i(u_1,\ldots,u_n,0^N)<_{\rm lex} (\varepsilon_1(\beta_2),\ldots,\varepsilon_{N}(\beta_2))$ for any $1\leq i<n+N$ which implies that $(u_1,\ldots,u_n,0^N)$ is non-recurrent. Thus, by Lemma \ref{S}(2), we have
$$|I_n^P(u_1,\ldots,u_n)|\geq |I_{n+N}^P(u_1,\ldots,u_n,0^N)|\geq C(u_1,\ldots,u_n,0^N)\overline{\beta}(u_1,\ldots,u_n,0^N)^{-(n+N)}.$$
It follows from the fact $\overline{\beta}(u_1,\ldots,u_n,0^N)\leq \beta_2$ that $$|I_n^P(u_1,\ldots,u_n)|\geq C(\beta_1)\beta_2^{-(n+N)}.$$

\end{proof}

Let us now concentrate on giving the lower bound of $\dim_{\rm H}E(\beta_1,\beta_2)$. As the conventional process, we define a measure supported on $E(\beta_1,\beta_2)$ which is similar to Section 3.2 by distributing the mass uniformly. We will give the local dimension $\liminf\limits_{n\rightarrow \infty}\frac{\log \mu\left(I_n^P(u)\right)}{\log|I_n^P(u)|}$ for any cylinder $I_n^P(u)$ which has non-empty intersection with $E(\beta_1,\beta_2)$. Without any confusion, here and subsequently, $I_n^P$ stands for the cylinder $I_n^P(u)$ for all $u\in\Lambda_n$.

(1) Define a probability measure supported on $E(\beta_1,\beta_2)$. Let $$\mu([\beta_1,\beta_2))=1\quad {\rm and}\quad \mu(I^P_{n_1}(u))=\frac{1}{\sharp G'_1},\ {\rm for}\ u \in D'_1.$$ For all $k\geq 1,$ and $u=(u_1,\ldots,u_{k+1})\in D'_{k+1}$, define
$$\mu(I_{n_{k+1}}^P(u))=\frac{\mu(I_{n_k}^P(u_1,\ldots,u_k))}{\sharp G'_{k+1}}.$$

(2) Estimate the local dimension $\liminf\limits_{n\rightarrow \infty}\frac{\log \mu\left(I_n^P\right)}{\log|I_n^P|}$ where $I_n^P\cap E(\beta_1,\beta_2) \neq \emptyset$. It follows from the definition of the measure that
\begin{equation}\label{m'1}
\mu(I_{n_i}^P)=\frac{1}{q'_i}\leq \frac{1}{c'(\beta_1, \widetilde{\beta}_N)c(\beta_1, \widetilde{\beta}_N)^i \beta_1^{\sum\limits_{j=1}^{i-1}(n_{j+1}-m_j)}},
\end{equation}for every $i>k(\beta_1, \widetilde{\beta}_N).$ For any $\beta\in E(\beta_1,\beta_2)$, suppose $\varepsilon(1,\beta)=(u_1,u_2,\ldots)$. For each $n\geq 1$, there exists an integer $k\geq 1$ such that $n_k <n\leq n_{k+1}$. It falls naturally into three cases.

Case 1. $n_k<n\leq m_k$. It follows from (\ref{m'1}) that $$\mu(I_n^P)= \mu(I_{n_k}^P)\leq c'(\beta_1, \widetilde{\beta}_N)^{-1}c(\beta_1, \widetilde{\beta}_N)^{-k} \beta_1^{-\sum\limits_{j=1}^{k-1}(n_{j+1}-m_j)}.$$ Furthermore,  by the construction of $E(\beta_1,\beta_2)$, the word $(u_1,\ldots,u_{m_k})$ is non-recurrent. Thus, by Lemma \ref{S}, we have  $$|I_n^P|\geq |I_{m_k}^P|\geq C(u_1,\ldots,u_{m_k})\overline{\beta}(u_1,\ldots,u_{m_k})^{-{m_k}}\geq c(\beta_1)\beta_2^{-m_k}.$$
As a consequence, $$\frac{\log \mu(I_n^P)}{\log|I_n^P|}\geq \frac{\log c'(\beta_1, \widetilde{\beta}_N)+k \log c(\beta_1, \widetilde{\beta}_N)+{\sum\limits_{j=1}^{k-1}(n_{j+1}-m_j)\log \beta_1}}{\log c(\beta_1)+m_k\log \beta_1}.$$

Case 2. $n=m_k+i(m_k-n_k)+\ell$ for some $0\leq i < t_k$ and $0\leq \ell< m_k-n_k$. On the one hand, when $0\leq \ell\leq 2N$, we have
$$\begin{aligned}
\mu(I_n^P)= \mu(I_{m_k+i(m_k-n_k)+\ell}^P)&\leq \mu(I_{m_k}^P)\cdot\frac{1}{(\sharp M_{m_k-n_k})^i}\\
&\leq c'(\beta_1, \widetilde{\beta}_N)^{-1}c(\beta_1,\widetilde{\beta}_N)^{-k} {\beta_1}^{-\left(\sum\limits_{j=1}^{k-1}(n_{j+1}-m_j)+i(m_k-n_k)\right)}.
\end{aligned}$$
On the other hand, when $2N< \ell< m_k-n_k$, we have
$$\begin{aligned}
\mu(I_n^P)= \mu(I_{m_k+i(m_k-n_k)+\ell}^P)&\leq \mu(I_{m_k}^P)\cdot\frac{1}{(\sharp M_{m_k-n_k})^i}\cdot\frac{1}{\Sigma_{\widetilde{\beta}_N}^{\ell-2N}}\\
&\leq c'(\beta_1, \widetilde{\beta}_N)^{-1}c(\beta_1, \widetilde{\beta}_N)^{-k}{\beta_1}^{-\left(\sum\limits_{j=1}^{k-1}(n_{j+1}-m_j)+i(m_k-n_k)\right)}{\widetilde{\beta}_N}^{-\ell+2N}.
\end{aligned}$$
Moreover, by Lemma \ref{i},  $$|I_n^P|=|I_{m_k+i(m_k-n_k)+\ell}^P|\geq C(\beta_1){\beta_2}^{-(m_k+i(m_k-n_k)+\ell+N)}.$$
Hence, $$\begin{aligned}
&\frac{\log \mu(I_n^P)}{\log|I_n^P|}\\&\geq \frac{\left(\sum\limits_{j=1}^{k-1}(n_{j+1}-m_j)+i(m_k-n_k)\right) \log \beta_1+(\ell-2N)\log\widetilde{\beta}_N+k\log c(\beta_1, \widetilde{\beta}_N)+\log c'(\beta_1, \tilde{\beta}_N)}{\log C(\beta_1)+(m_k+i(m_k-n_k)+\ell+N) \log \beta_2}.
\end{aligned}$$

Case 3.  $n=m_k+t_k(m_k-n_k)+\ell$ for some $0\leq \ell\leq p_k$. Similarly, when $0\leq \ell\leq 2N$, we have
$$\begin{aligned}
\mu(I_n^P)= \mu(I_{m_k+i(m_k-n_k)+\ell}^P)&\leq \mu(I_{m_k}^P)\cdot\frac{1}{(\sharp M_{m_k-n_k})^{t_k}}\\
&\leq c'(\beta_1, \widetilde{\beta}_N)^{-1}c(\beta_1,\widetilde{\beta}_N)^{-k} {\beta_1}^{-\left(\sum\limits_{j=1}^{k-1}(n_{j+1}-m_j)+t_k(m_k-n_k)\right)}.
\end{aligned}$$
When $2N< \ell\leq p_k$, it follows that
$$\begin{aligned}
\mu(I_n^P)= \mu(I_{m_k+t_k(m_k-n_k)+\ell}^P)&\leq \mu(I_{m_k}^P)\cdot\frac{1}{(\sharp M_{m_k-n_k})^i}\cdot\frac{1}{\Sigma_{\widetilde{\beta}_N}^{\ell-2N}}\\
&\leq c'(\beta_1, \widetilde{\beta}_N)^{-1}c(\beta_1, \widetilde{\beta}_N)^{-k}{\beta_1}^{-\left(\sum\limits_{j=1}^{k-1}(n_{j+1}-m_j)+i(m_k-n_k)\right)}{\widetilde{\beta}_N}^{-\ell+2N}.
\end{aligned}$$
Furthermore, we conclude from Lemma \ref{i} that $$|I_n^P|= |I_{m_k+t_k(m_k-n_k)+\ell}^P|\geq C(\beta_1){\beta_2}^{-(m_k+t_k(m_k-n_k)+\ell+N)}.$$
Therefore, we have $$\begin{aligned}
&\frac{\log \mu(I_n^P)}{\log|I_n^P|}\\&\geq \frac{\left(\sum\limits_{j=1}^{k-1}(n_{j+1}-m_j)+t_k(m_k-n_k)\right) \log \beta_1+(\ell-2N)\log\tilde{\beta}_N+k\log c(\beta_1, \widetilde{\beta}_N)+\log c'(\beta_1, \widetilde{\beta}_N)}{\log C(\beta_1)+(m_k+t_k(m_k-n_k)+\ell+N) \log \beta_2}.
\end{aligned}$$

Just proceeding as the same analysis in Section 3.2, for all the above three cases, we obtain   $$\liminf_{n\rightarrow \infty}\frac{\log \mu(I_n^P)}{\log|I_n^P|}\geq \left(1-\frac{b^2(1-a)}{b-a}\right)\frac{\log\beta_1}{\log\beta_2}.$$

(3) Use the mass distribution principle (see \cite[Page 60]{FE}).  Now we take any $B(\beta,r)$ with center $\beta\in E(\beta_1,\beta_2)$ and sufficiently small enough $r$ verifying
\begin{equation}\label{I1}
|I_{n+1}^P|\leq r<|I_n^P|\leq \beta_1^{-n+1},
\end{equation}where the last inequalities is guaranteed by the fact that $\overline{\beta}(\omega)\geq \beta_1$ for any $\omega\in \Lambda_n(\beta_1,\beta_2)$. By Lemma \ref{i}, we have $$|I_n^P|\geq C(\beta_1)\beta_2^{-(n+N)}.$$ As a result, the ball $B(\beta,r)$ intersects no more than $2\left\lfloor C(\beta_1)^{-1}\beta_2^N\left(\frac{\beta_2}{\beta_1}\right)^{n-1}\right\rfloor+2$ cylinders of order $n$. Moreover, it follows from Lemma \ref{i} that \begin{equation}\label{I2}
r\geq|I_{n+1}^P|\geq C(\beta_1)\beta_2^{-(n+1+N)}.
\end{equation} Immediately, the combination of (\ref{I1}) and (\ref{I2}) gives
$$\begin{aligned}
&\liminf_{r\rightarrow 0}\frac{\log \mu\left(B(\beta,r)\right)}{\log r}\\ &\geq \liminf_{n\rightarrow\infty}\frac{\log\left(2\left\lfloor C(\beta_1)^{-1}\beta_2^N\left(\frac{\beta_2}{\beta_1}\right)^{n-1}\right\rfloor+2\right)+\log\mu\left(I_n^P\right)}{\log |I_n^P|}\cdot\frac{\log |I_n^P|}{-\log C(\beta_1)+(n+1+N)\log \beta_2}\\
&\geq \liminf_{n\rightarrow\infty}\left(\frac{(n-1)(\log\beta_2-\log \beta_1)}{-\log C(\beta_1)+ (n+N)\log \beta_2}+\frac{\log \mu\left(I_n^P\right)}{\log |I_n^P|}\right)\cdot\frac{(n-1)\log \beta_1}{-\log C(\beta_1)+ (n+1+N)\log \beta_2}\\ &\geq \left(\frac{\log\beta_2-\log \beta_1}{\log\beta_2}+\left(1-\frac{b^2(1-a)}{b-a}\right)\frac{\log\beta_1}{\log\beta_2}\right)\frac{\log\beta_1}{\log\beta_2}.
\end{aligned}$$
Therefore, by the mass distribution principle and letting $\beta_1\rightarrow \beta_2,$ we get our desired result.
\section{Proof of Theorem \ref{res'}}
Akin to Section 5, we need to find a subset of $E^P_{0,1}$ which is a dense $G_\d$ set in the interval $[1,2]$. Since the process of our proof is almost the same as Section 5. We only provide the construction of the required set $V'$ in this section.

For all $k\geq 1$, we first choose the sequences $\{n_k\}_{k=1}^\infty$ and $\{m_k\}_{k=1}^\infty$  such that $m_k-n_k>\max\{2(m_{k-1}-n_{k-1}), n_k-k\}$ and  $n_k<m_k<n_{k+1}$. In addition, the sequences  $\{n_k\}_{k=1}^\infty$ and $\{m_k\}_{k=1}^\infty$ is chosen to satisfy $$\lim_{k\rightarrow \infty}\frac{m_k-n_k}{n_{k+1}+m_k-n_k}=0,$$and$$\lim_{k\rightarrow \infty}\frac{m_k-n_k}{m_k}=1.$$
Actually, let $$n'_k=k^{2k}\quad {\rm and}\quad  m'_k=(k+1)^{2k+1}.$$ We can obtain the required sequences with some adjustments.

For all $k\geq 1$, denote $n_{k+1}=(m_k-n_k)t_k+m_k+p_k$ where $0\leq p_k <m_k-n_k$. We now define$$V'=\bigcap_{n=1}^\infty\bigcup_{k=n}^\infty\bigcup_{(\epsilon_1,\ldots,\epsilon_k) \in\Lambda_k(1,2)}{\rm int} \left(I_{n_{k+1}}^P\left(\epsilon_1,\ldots,\epsilon_k ,0^{n_k-k},(1,0^{m_k-n_k})^{t_k},0^{p_k}\right)\right)$$ where $\Lambda_k(1,2)$ is defined by (\ref{la}). Since $m_k-n_k>n_k-k$, we get $(1,0^{m_k-n_k})<_{\rm lex}(\varepsilon_1,\cdots,\varepsilon_k,0^{n_k-k})$ for all $k\geq 1$. By Lemma \ref{se}, the set $V'$ is well defined.

{\bf Acknowledgement} This work was partially supported by NSFC 11771153.

Author's E-mail address: {\small \it  lixuan.zheng@u-pec.fr}
\end{document}